\documentclass[12pt, twoside]{amsart}
\usepackage[latin1]{inputenc}
\usepackage{times}
\usepackage{amsthm}
\usepackage{amssymb, verbatim}
\usepackage{amsmath}
\usepackage[outline]{contour}
\usepackage{mathrsfs}
\usepackage[all]{xy}
\usepackage{lscape}
\usepackage{tikz}
\usetikzlibrary{decorations.pathmorphing}
\usetikzlibrary{arrows}
\usepackage{graphicx}
\usepackage{pstricks}
\usepackage{caption}
\usepackage{centernot}
\usepackage{stmaryrd}

\topmargin = - 40 pt 
\usepackage{color}

\newcommand     {\Nor}    {{\mathscr{N}}}

\newtheorem{lemma}{Lemma}[section]

\theoremstyle{definition}

\newtheorem{defn}{Definition}[section]

\newcommand{\thref}[1]{Theorem {\rm\ref{#1}}}

\newcommand{\reref}[1]{Remark {\rm\ref{#1}}}

\def\<{\langle}
\def\>{\rangle}

\newcommand{\bpm}{\begin{pmatrix}}

\newcommand{\epm}{\end{pmatrix}}

\numberwithin{equation}{section}

\def\HK{\mathbf{HK}}

\newcommand{\io}{\iota}

\newcommand{\thickto}{\pmb{\contour{black}{$\,\longrightarrow\,$}}}
\newcommand{\notthickto}{\pmb{\contour{black}{$\,\centernot\longrightarrow\,$}}} 
\newcommand{\rto}{\stackrel{r}{\longrightarrow}}
\newcommand{\lto}{\stackrel{l}{\longrightarrow}}
\newcommand{\starto}{\stackrel{*}{\longrightarrow}}

\newtheorem{theorem}{Theorem}[section]
\newtheorem{definition}[theorem]{Definition}
\newtheorem{proposition}[theorem]{Proposition}
\newtheorem{corollary}[theorem]{Corollary}
\newtheorem{remark}[theorem]{Remark}
\newtheorem{example}[theorem]{Example}

\title[Normal form in Hecke-Kiselman monoids]{Normal form in Hecke-Kiselman monoids associated with simple oriented graphs}

\author{Riccardo Aragona}
\address{DISIM, University of L'Aquila, Via Vetoio, Coppito - 67100 L'Aquila, Italy}
\email{riccardo.aragona@univaq.it}

\author{Alessandro D'Andrea}
\address{Dipartimento di Matematica, Universit\`a degli Studi di
Roma ``La Sapienza''\\ P.le Aldo Moro, 5 -- 00185 Rome, Italy}
\email{dandrea@mat.uniroma1.it}

\thanks{Both authors are members of INdAM-GNSAGA. The second named author was partially supported by Ateneo fundings from Sapienza University in Rome.}

\begin{document}
\maketitle

\begin{abstract}
We generalize Kudryavtseva and Mazorchuk's concept of canonical form of elements \cite{KM09} in Kiselman's semigroups to the setting of a Hecke-Kiselman monoid $\HK_\Gamma$ associated with a simple oriented graph $\Gamma$. We use confluence properties from \cite{huet} to associate with each element in $\HK_\Gamma$ a normal form; normal forms are not unique, and we show that they can be obtained from each other by a sequence of {\em elementary commutations}. We finally describe a general procedure to recover a (unique) lexicographically minimal normal form.
\end{abstract}


\section{Introduction}


Let $\Gamma$ be a {\em simple mixed graph}, i.e., each pair of distinct vertices in $\Gamma$ has at most one connection, which can be either oriented or unoriented; thus, there are no oriented cycles of length two and 
no vertex in $\Gamma$ has a self loop.

One may use \cite{GM11} the combinatorial content of $\Gamma$ to give a presentation of a Hecke-Kiselman semigroup $\HK_\Gamma$. If $V$ is the set of vertices of $\Gamma$, $\HK_\Gamma$ is generated by idempotent elements $a_i, i \in V$, which satisfy the relations:
\begin{itemize}
\item $a_i a_j = a_j a_i$ if there is no connection between the vertices $i, j \in V$;
\item $a_i a_j a_i = a_j a_i a_j$ if there is an unoriented connection between $i$ and $j$;
\item $a_i a_j a_i = a_j a_i a_j = a_i a_j$ if there is an arrow connecting $i$ to $j$.
\end{itemize}

If $\Gamma$ is an unoriented simple graph, i.e., if relations in the above presentation are all of the first two kinds, then we obtain the Coxeter monoid associated to the simply laced Dynkin diagram $\Gamma$. This is also known in the literature either as Richardson-Springer \cite{RS} or $0$-Hecke monoid \cite{0-Hecke}, as its monoid algebra \cite{norton} may be obtained as the $q=0$ specialization of a Iwahori-Hecke algebra.

The third type of relation has been first observed by Kiselman \cite{kiselman}. When $\Gamma = \Gamma_n$ is the graph on the vertex set $\{1, 2, \dots, n\}$ with a single oriented connection between $i$ and $j$ each time that $i < j$, one obtains the so-called {\em Kiselman semigroups}, so that Kiselman's original example corresponds to $\Gamma_3$. These semigroups also occur in the study \cite{cd15} of some graph-dynamical systems related to SDS \cite{sds}.

Understanding which mixed graphs $\Gamma$ yield finite Hecke-Kiselman monoids is a difficult problem and the only nontrivial results so far seem to be \cite{ad13} and \cite{aj, aj2}. In the same vein, a characterization of reduced expressions of elements as words in the idempotent generators are only known in the Kiselman case $\Gamma = \Gamma_n$ \cite{KM09} or when $\Gamma$ is an unoriented graph and one may reduce to standard Coxeter combinatorics.
The present paper deals with the easier case where only oriented connections occur. We should stress that our result is implicit in \cite{GM11} when $\Gamma$ is an equioriented Dynkin graph of type $A_n$.

We employ Huet's reformulation \cite{huet} of Newman's results \cite{new} to extend the strategy outlined by Kudryavtseva and Mazorchuk \cite{KM09} for Kiselman's semigroups, to all (possibly infinite) Hecke-Kiselman monoids corresponding to simple oriented graphs. The concrete statement is that {\em normal forms} of each element in $\HK_\Gamma$ all arise via a decreasing sequence of cancellations, that all decreasing sequences of cancellations may be continued to a normal form, that such normal forms all have the same length, and may be obtained from each other by a sequence of elementary commutations between pairs of disconnected idempotent generators. We end the paper with some final comments on how to select a lexicographically minimal normal form.

\section{Normal Form in $\HK_{\Gamma}$}

In what follows, $\Gamma=(V,E)$ will be a {\em simple oriented  graph}, i.e., a directed graph that does not have oriented cycles of length 1 or 2, so that there are no self-loops and there is at most one connection between two given vertices. Here $V$ denotes the vertex set and $E\subseteq V\times V$ is the arrow set, where $(a,b)\in E$ if and only if there is an arrow connecting $a$ to $b$; indeed, we will use the shorthand notation
 $a \thickto b$ as equivalent to $(a,b) \in E$. Notice, however, that we will reserve the symbol $\longrightarrow$ for a different context, in order to adhere to notations from \cite{huet}.

Any given choice of $\Gamma$ yields a {\em Hecke-Kiselman monoid} $\HK_\Gamma$ defined by the presentation
\begin{align*}
\HK_{\Gamma}=\langle a\in V \;|\; & a^2=a, \text{ for every } a \in V;\\
& a b a = b a b = a b, \text{ if } a\thickto b;\\
& a b = b a, \text{ if } a \notthickto b \text{ and } b \notthickto  a \rangle.
\end{align*}

If we denote by $F(V)$ the free monoid on the alphabet $V$, then we have a canonical projection
$$
\pi:F(V) \to \HK_{\Gamma}.
$$
Every $a\in V\subseteq F(V)$ will  be called a \emph{letter}; if $w\in F(V)$ is obtained by multiplying letters among which $a$ occurs, we will say that (the word) $w$ contains (the letter) $a$, or that $a$ occurs in $w$. The same terminology will be used when $w\in \HK_{\Gamma}$; this is well defined as the letter content in both sides of each relation presenting $\HK_{\Gamma}$ is the same, so that all words in $F(V)$ projecting via $\pi$ to the same element of $\HK_\Gamma$ have the same letter content. Note that each letter in $\HK_{\Gamma}$ is idempotent.

\begin{remark}
$\HK_{\Gamma}$ is finite if and only if $\Gamma$ is acyclic \cite{ad13}.
\end{remark}

\begin{lemma}\label{lem:awa=aw}
If $a\in \HK_\Gamma$ is a letter and $w\in\HK_{\Gamma}$ is obtained by multiplying  letters that do not admit arrows to (respectively from) $a$, then $awa=aw$ (resp. $awa=wa$).
\end{lemma}
\begin{proof}
If $a$ and $b$  are letters in $\HK_{\Gamma}$ such that $b\notthickto a$, then either $a \thickto b$, whence $aba=ab$, or $ab=ba$, whence $aba = a(ba) = a(ab) = a^2b = ab$. Also, if $u$ and $v$ are words in $\HK_{\Gamma}$ satisfying respectively $aua = au$ and  $ava = av$, then  $auva = (au)va = (aua)va = au(ava) = au(av) = (aua)v = auv$. Now the statement follows by an easy induction.\\
The case where $w$ is obtained by multiplying letters that do not admit arrows from $a$ is done similarly.
\end{proof}

Let $w_1, w_2, u \in F(V)$. It is useful to introduce the following  \emph{elementary cancellations} on words in $F(V)$.
\begin{itemize}
\item \emph{Right cancellation}: $w_1 a u a w_2\rto  w_1 a u w_2$, if $a$ is a letter and no letter in $u$ has an arrow to $a$;
\item \emph{Left cancellation}: $w_1 a u a w_2\lto  w_1  u a w_2$, if $a$ is a letter and no letter in $u$ has an arrow from $a$.
\end{itemize}
Without loss of generality, we may assume above that $u$ does not contain the letter $a$ and only focus on elementary cancellations between {\em consecutive} occurrences of the same letter. 
Notice that if $v$ is obtained from $w$ by a sequence of elementary cancellations, then $v,w\in F(V)$ map to the same element in $\HK_{\Gamma}$.

\begin{remark}\label{altpresentation}
Idempotence of letters and each relation $aba = bab = ab$ in the presentation of $\HK_\Gamma$ are special instances of elementary cancellations. Thus, elementary cancellations along with commutations of disconnected letters provide an equivalent presentation of $\HK_\Gamma$.
\end{remark}

\begin{definition}
Let $v,w\in F(V)$. We shall write by $w \starto v$ if $v$ is obtained from $w$ by a (possibly empty) sequence of (either right or left) elementary cancellations. In other words, $\starto$ is the reflexive-transitive closure of the relation $\longrightarrow\, = \,\rto \cup \lto$ on $F(V)$.
\end{definition}
A  \emph{simplifying sequence} (with respect to $\Gamma$) from $w$ to $v$ is a sequence of elementary cancellations which transform $w$ into $v$. This is is analogous to \cite[Remark 7]{KM09}.
We take the following definition from \cite{huet}
.
\begin{definition}
A word $w\in F(V)$  is a \emph{normal form} for $\pi(w)\in\HK_{\Gamma}$  if no elementary cancellation may be performed  on $w$.
\end{definition}

We shall denote by $\mathscr{N}$ the set of all normal forms in $F(V)$. Notice that $\mathscr{N}$ depends on $\Gamma$, which we consider to be fixed once and for all.

\begin{remark}
By Lemma \ref{lem:awa=aw}, $w \in F(V)$ is a normal form for $\pi(w)\in\HK_\Gamma$ if and only if each subword of $w$ of the form $aua$, where $a\in V$
and $u\in F(V)$, contains at least one letter with an arrow to $a$ and at least one letter with an arrow from $a$. In \cite{GM11} these words are called {\em special} when $\Gamma = \Gamma_n$ and {\em strongly special} when $\Gamma$ is an equioriented Dynkin diagram of type $A_n$.
\end{remark}


Note that, by definition, if $\gamma\in\HK_{\Gamma}$, then any word $w\in \pi^{-1}(\gamma)$ of minimal length is a normal form of $\gamma$ so, in particular, each $\gamma\in\HK_{\Gamma}$ admits at least one normal form. However, in principle, a normal form of $\gamma\in\HK_{\Gamma}$ may fail to be of minimal length. We will show that this is not the case by proving that all normal forms of $\gamma$ share the same length and, more precisely, that they can be obtained from each other by a sequence of commutations between disconnected letters. 

Recall that 
\cite[Theorem 6]{KM09} exploits Newman's Diamond Lemma~\cite{new} in the case of the complete oriented acyclic $\Gamma_n$, so as to show that:
\begin{enumerate}
\item each $\gamma\in\HK_{\Gamma_n}$ has a unique normal form;
\item every word $w\in\pi^{-1}(\gamma)$ is connected  to the unique normal form for $\gamma$ by a simplifying sequence;
\item each simplifying sequence starting from $w$ may be completed to a sequence as in (2).
\end{enumerate}
Claim (1) may certainly fail in our generalized setting. Indeed if $a,b\in\HK_{\Gamma}$ are commuting letters then $ab$ and $ba$ are distinct normal forms for the same element in $\HK_{\Gamma}$. We want to show that this is basically the only obstruction to uniqueness.

\begin{definition}
We  denote by $\sim$ the equivalence relation on  $F(V)$ generated by  \emph{elementary commutations}
$$
w_1abw_2 \sim w_1ba w_2,
$$
where $w_1,w_2\in F(V)$ and $a,b\in V$ are \emph{disconnected} letters, i.e., they satisfy $a\notthickto b$ and $b \notthickto a$.
\end{definition}

Our strategy is to use confluence properties 
of the relation $\longrightarrow$ modulo the equivalence $\sim$ on $F(V)$.  
In order to do so, we set ourselves within the framework described by Huet in \cite[Section 2.3]{huet} 
 to make sure that the possibility to apply any given elementary cancellation on $w\in F(V)$ only depends on its $\sim$-equivalence class; furthermore that such elementary cancellations yield $\sim$-equivalent words.

\begin{lemma}\label{lem:simcom}
Let $v, w \in F(V)$, and assume that $w\starto v$. If $\widetilde{w}\sim w$, then there exists $\widetilde{v}\sim v$ such that 
$\widetilde{w}\starto \widetilde{v}$.
\end{lemma}
\begin{proof}
It suffices to only treat the case where $\widetilde{w}$ is obtained from $w$ by a single elementary commutation and $v$ is obtained from $w$ by a single elementary cancellation.

Set $w=w_1 a u a w_2$ and $v=w_1 a u w_2$, where no letter in $u$ has an arrow to $a$.
Let $\widetilde{w}$ be obtained from $w$ by means of an elementary commutation; if this occurs outside $aua$ or within $u$, then the claim is clear. The only possibly nontrivial case is when the elementary commutation involves either the leading or the ending letter in subword $aua$. Without loss of generality we may assume that $u$ does not contain the letter $a$ and the elementary commutation involves a letter of $u$.
We have two cases:
\begin{enumerate}
\item[(i)] $u=bu'$ and the elementary commutation involves $a$ and $b$. Then $w=w_1abu'aw_2$ and $\widetilde{w}=w_1bau'aw_2$.  Notice that as $u'$ is a subword of $u$, no letter from $u'$ has an arrow to $a$. Thus we may perform a right cancellation on the subword $au'a$ giving $\widetilde{v}=w_1bau'w_2$. However, $\widetilde{v}$ is obtained from $v = w_1abu'w_2$ by elementary commutation of $a$ with $b$.

\item[(ii)] $u=u'b$ and the elementary commutation involves $b$ and $a$. Then $w=w_1au'baw_2$ and $\widetilde{w}=w_1au'abw_2$. Once more we may perform
 a right cancellation on the subword $au'a$ giving $\widetilde{v}=w_1au'bw_2$, which  coincides with $v$. 
%
%
%
\end{enumerate}
The proof for left cancellations is completely analogous.
\end{proof}

\begin{remark}\label{rem:uptocom}
\begin{itemize}
\item By Lemma~\ref{lem:simcom}, each element in the $\sim$-equivalence class of a normal form is also a normal form. 
\item If $u$ only contains letters that are not connected to the letter $a$, then both a right and a left cancellation may be performed on $w_1 aua w_2$. However, the resulting words $w_1 au w_2$ or $w_1 ua w_2$ lie in the same $\sim$-equivalence class. 
\end{itemize}
\end{remark}



\section{Normal Forms and confluence}

Let us consider the framework of \cite[Section 2.3]{huet}, where, in our setting, 
\begin{itemize}
\item $\mathscr{E}$  is the free monoid $F(V)$ on $V$, 
\item $\longrightarrow$ is the binary relation $\rto \cup \lto$ on $F(V)$,
\item $\starto$ is the reflexive-transitive closure of the relation $\longrightarrow$ on $F(V)$,  
\vspace{1.5mm}
\item $\sim$ is the equivalence relation on $F(V)$ generated by elementary commutations of disconnected letters, and
\vspace{1.5mm}
\item 
$x \equiv y$ if and only if $\pi(x)=\pi(y)$, where $\pi: F(V) \to \HK_\Gamma$ is the canonical projection. Indeed, by \reref{altpresentation}, the equivalence relation $\equiv$ generated by $\longrightarrow \cup \sim$,  as from \cite[Lemma 2.6]{huet}, coincides with the quotient relation induced by the presentation of $\HK_\Gamma$.
\end{itemize}
\begin{defn}[\cite{huet}]
\begin{itemize}
The relation $\longrightarrow$ on $F(V)$ is
\item
{\em confluent modulo $\sim$} iff for all choices of $x \sim y, x', y' \in F(V)$ such that $x \starto x', y \starto y'$, one may find $\overline{x}, \overline{y}$ such that
$$x' \starto \overline{x}, \qquad y' \starto \overline{y}, \qquad \overline{x} \sim \overline{y}.$$
\item
{\em locally confluent modulo $\sim$} iff the following conditions are satisfied
\begin{itemize}
\item[$\alpha$:] for all $x,y,z\in F(V)$ such that $y$ and $z$ are obtained from $x$ by any elementary cancellation, then there exist $u,v\in F(V)$ such that $y \starto u$, $z \starto v$ and $u\sim v$;
\item[$\beta$:]  for all $x,y,z\in F(V)$ such that $x \sim y$ and $z$ is obtained from $x$ by any elementary cancellation, then there exist $u,v\in F(V)$ such that $y \starto u$, $z \starto v$ and $u\sim v$.
\end{itemize}
\end{itemize}
\end{defn}

\begin{center}
\begin{minipage}{0.4\textwidth}
\begin{center}
\begin{tikzpicture}
\tikzset{vertex/.style }
\tikzset{edge/.style = {->,> = latex'}}
\node[vertex] (x) at  (0,0) {$x$};
\node[vertex] (y) at  (-2,-2) {$y$};
\node[vertex] (z) at  (2,-2) {$z$};
\node[vertex] (u) at  (-1,-4) {$u$};
\node[vertex] (v) at  (1,-4) {$v$};
\draw[edge] (x) to (z);
\draw[edge] (x) to (y);
\draw[->] (y) to node[left] {$*$} (u);
\draw[->] (z) to node[right] {$*$} (v);
\draw[bend left=45] (u) to (0,-4);
\draw[bend right=45] (0,-4) to (v);
\end{tikzpicture}\par\medskip
\small{Condition $\alpha$}
\end{center}
\end{minipage}
\begin{minipage}{0.4\textwidth}
\begin{center}
\begin{tikzpicture}
\tikzset{vertex/.style }
\tikzset{edge/.style = {->,> = latex'}}
\node[vertex] (x) at  (0,0) {$x$};
\node[vertex] (y) at  (-2,0) {$y$};
\node[vertex] (z) at  (2,-2) {$z$};
\node[vertex] (u) at  (-2,-4) {$u$};
\node[vertex] (v) at  (0,-4) {$v$};
\draw[edge] (x) to (z);
\draw[bend left=45] (x) to (-1,0);
\draw[bend right=45] (-1,0) to (y);
\draw[->] (y) to node[left] {$*$}  (u);
\draw[->] (z) to node[right] {$*$}  (v);
\draw[bend left=45] (u) to (-1,-4);
\draw[bend right=45] (-1,-4) to (v);
\end{tikzpicture}\par\medskip
\small{Condition $\beta$}
\end{center}
\end{minipage}
\end{center}

Since each elementary cancellation decreases word length, the relation $\longrightarrow$ is noetherian \cite[Section 2.1]{huet}, i.e., there is no infinite sequence of elementary cancellations. In the noetherian case, \cite[Lemma 2.7]{huet} shows that confluence modulo $\sim$ and local confluence modulo $\sim$ are equivalent.


The following theorem proves that each normal form of a word $w\in\pi^{-1}(\gamma)$ belongs to the same $\sim$-equivalence class  and it follows that all the normal forms of $w$ have the same length. In particular we obtain that  the number of simplifying steps to achieve a normal form of $\pi(w)$ starting from $w$ is independent of the chosen simplifying sequence.

\begin{theorem}\label{th:main}
\begin{enumerate}
\item Let $x,y\in F(V)$. If $\pi(x)=\pi(y)$ and $u,v \in \Nor$ satisfy $x\starto u$ and $y\starto v$, then $u\sim v$.
\item Every simplifying sequence
$$x \longrightarrow x_1 \longrightarrow x_2 \dots \longrightarrow x_n$$
may be extended to a simplifying sequence ending on a normal form of $\pi(x)$.
\item All simplifying sequences starting from $x\in F(V)$ and ending on some normal form of $\pi(x)$ have the same length.
\end{enumerate}
\end{theorem}
\begin{proof}
The second claim is a rephrasing of the concept of normal form, whereas the third claim follows immediately from the first one, once we notice that words in the same $\sim$-equivalence class have the same length and each elementary cancellation decreases word length by exactly one.

As for the first claim, this is just \cite[Lemma 2.6]{huet}, which is equivalent to $\longrightarrow$ being confluent modulo $\sim$. 
As we are in a noetherian setting, it is enough to prove that local confluence holds, i.e., that conditions $\alpha$ and $\beta$ are satisfied.

Condition $\alpha$ follows from \cite[Lemma 8]{KM09}, where Kudryavtseva and Mazorchuk, more in general, prove that for all $x,y,z\in F(V)$ such that $y$ and $z$ are obtained from $x$ by any elementary cancellation, then there exists $u\in F(V)$ such that $y \starto u$ and $z \starto u$.

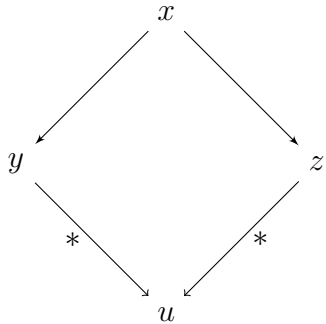
\begin{center}
\begin{tikzpicture}
\tikzset{vertex/.style }
\tikzset{edge/.style = {->,> = latex'}}
\node[vertex] (x) at  (0,0) {$x$};
\node[vertex] (y) at  (-2,-2) {$y$};
\node[vertex] (z) at  (2,-2) {$z$};
\node[vertex] (u) at  (0,-4) {$u$};
\draw[edge] (x) to (z);
\draw[edge] (x) to (y);
\draw[->] (y) to node[left] {$*$} (u);
\draw[->] (z) to node[right] {$*$} (u);
\end{tikzpicture}
\end{center}

Condition $\beta$ follows from the fact that, by Lemma \ref{lem:simcom}, for all $x,y,z\in F(V)$ such that $x \sim y$ and $z$ is obtained from $x$ by any elementary cancellation, there exists $u \in F(V)$ such that $z \sim u$ and $u$ is obtained from $y$ by an elementary cancellation.
\end{proof}

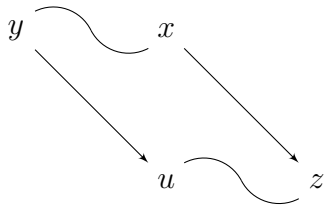
\begin{center}
\begin{tikzpicture}
\tikzset{vertex/.style }
\tikzset{edge/.style = {->,> = latex'}}
\node[vertex] (x) at  (0,0) {$x$};
\node[vertex] (y) at  (-2,0) {$y$};
\node[vertex] (z) at  (2,-2) {$z$};
\node[vertex] (u) at  (0,-2) {$u$};
\draw[edge] (x) to (z);
\draw[edge] (y) to (u);
\draw[bend left=45] (y) to (-1,0);
\draw[bend right=45] (-1,0) to (x);
\draw[bend left=45] (u) to (1,-2);
\draw[bend right=45] (1,-2) to (z);
\end{tikzpicture}
\end{center}

\begin{corollary}
If $u$ and $v$ are normal forms of the same element in $\HK_\Gamma$, then $u \sim v$. In particular, normal forms in $\HK_{\Gamma_n}$ are unique \cite{KM09}.
\end{corollary}

\begin{remark}
We stress the fact that \thref{th:main} uses neither finiteness of the graph $\Gamma$ nor that of the monoid $\HK_\Gamma$. For instance, when $\Gamma$ has an oriented cycle, one may prove that $\HK_\Gamma$ is infinite \cite{ad13} by noticing that each power of the ordered product of all letters in the cycle is a normal form, hence they describe infinitely many distinct elements, as they have distinct length.
\end{remark}



\section{Choosing a preferred normal form}
In actual contexts one would like to locate a \emph{favorite} normal form to work with. One way to do this is by choosing a total ordering $<$ on the set $V$ of the vertices of $\Gamma$ and employ the induced lexicographic ordering on $F(V)$ so as to choose the minimal normal form. 

If $[w]$ is the $\sim$-equivalence class of a normal form for some $\gamma\in \HK_{\Gamma}$, we will henceforth denote  by $w^{\min}$ its lexicographically minimal element. In principle, one may not be able to obtain $w^{\min}$ from $w$ by a sequence of lexicographically decreasing elementary commutations. 
\begin{example}
Consider the total ordering $a<b<c$ on the graph 

%
%

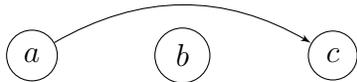
\begin{center}
\begin{tikzpicture}
\tikzset{vertex/.style = {shape=circle,draw,minimum size=1em}}
\tikzset{edge/.style = {->,> = latex'}}
\node[vertex] (a) at  (0,0) {$a$};
\node[vertex] (b) at  (2,0) {$b$};
\node[vertex] (c) at  (4,0) {$c$};
\draw[edge] (a) to [bend left] (c);
\end{tikzpicture}
\end{center}


If $w=cab$, then $[w]=\{bca,cab,cba\}$ so that $w^{\min}= bca$. Elementary commutations all involve $b$, so that the only way to commute $w$ into $w^{\min}$ is $cab\sim cba \sim bca$; however $cab$ is lexicographically lower than $cba$.
\end{example}
We thus need to find a general strategy to recover $w^{\min}$ from $w$. 
\begin{definition}
Let $a_i\in V$, $i=1,\ldots, n$, so that $w=a_1 a_2 \ldots a_n\in F(V)$. Then $a_j$ is an \emph{initial letter} of $w$, if $a_j$ commutes with $a_i$ for each $i<j$. 
\end{definition}

%

Denote now by $\io(w)$ the least initial letter of $w$
. We are going to describe a procedure to select a lexicographically minimal normal form for any element in $\HK_\Gamma$.
\begin{proposition}\label{lem:concan}
Let $w\in F(V)$ be a normal form. If $w^{\min}=a_1a_2\ldots a_n$, $a_i\in V$, then for all $k\geq 0$, $a_{k+1}$ is the least initial letter of the word $w_k$ obtained from $w$ by removing the leftmost occurrences of the letters $a_1,a_2,\ldots,a_k$. Equivalently, $w^{\min}=\io(w)\,{w_1}^{\min}$.
\end{proposition}

\begin{proof}
First of all, by the very definition, initial letters of $w$ all commute with which other. Also, if $w'$ is obtained from $w$ by an elementary commutation then the sets of initial letters of $w$ and $w'$ coincide; therefore, the set of initial letters of a word only depends on its $\sim$-equivalence class.\\
Every initial letter of $w$ may be commuted to the leftmost position; vice versa the leading letter of each word in the $\sim$-equivalence class of $w$ is an initial letter for $w$. Thus $\io(w)$ is the leftmost letter of $w^{\min}$. Now, $w^{\min}=\io(w)\,{w_1}^{\min}$ can be easily proved by induction on the length $n$ of $w$.
\end{proof}

\begin{corollary}
Define inductively a word $w \in F(V)$ to be {\em tidy} as follows:
\begin{itemize}
\item the empty word is tidy;
\item $w$ is tidy if $w = \io(w) w_1$ and $w_1$ is tidy.
\end{itemize}
Then there exists a bijection between $\HK_\Gamma$ and the set of tidy normal forms, which associates with every element $\gamma \in \HK_\Gamma$ its unique lexicographically minimal normal form.
\end{corollary}

The above claims shows that $w^{\min}$ can be recursively computed  from $w$. Notice, however, that the actual computation of $w^{\min}$ will strongly depend on the topology of $\Gamma$. For instance, when $\Gamma=\Gamma_n$, no elementary commutations will be needed at all, so that $w$ and $w^{\min}$ will always coincide. The opposite extreme is when $\Gamma$ is a totally disconneted graph; in this case a word is normal form if and  only if each of its letters  only occurs once. One then obtains $w^{\min}$ from $w$ by sorting $w$ with respect to the chosen total order. Intermediate cases will require some ``partial sorting'' of the letter content of $w$. How to do this efficiently seems to be an interesting problem, which will be addressed in a future paper.




\end{document}